\documentclass [11pt,a4paper, amsppt, reqno]{amsart}
\usepackage[colorlinks]{hyperref}

\input amssymb.sty

\DeclareMathOperator{\cok}{coker}

\newcommand{\Shyr}{\text{Shyr}}
\newcommand{\Ono}{\text{Ono}}

\newcommand{\fp}{\frak{p}}
\newcommand{\fP}{\frak{P}}

\newcommand{\fg}{\frak{g}}
\newcommand{\Hom}{\text{Hom}}

\newcommand{\Sp}{\text{Spec} \,}

\newcommand{\sep}{\text{sep}}

\chardef\bslash=`\\ % p. 424, TeXbook
\newtheorem{theorem}{Theorem}[section]

\newtheorem*{ack*}{Acknowledgments}
\newtheorem{prop}[theorem]{Proposition}
\newtheorem{lem}[theorem]{Lemma}
\newtheorem{cor}[theorem]{Corollary}

\theoremstyle{definition}
\newtheorem{defn}[theorem]{Definition}  % was Def..

\newtheorem{remark}[theorem]{Remark}
\newtheorem{example}[theorem]{Example}

\numberwithin{equation}{section}

\newtheorem*{maintheorem*}{Main Theorem}
\theoremstyle{definition}
\newtheorem{definition}{Definition}

\newtheorem{notation}[theorem]{Notation}

\newcommand{\CO}{\mathcal{O}}

\renewcommand{\sectionmark}[1]{}

\newcommand{\N}{\mathcal{N}}
\newcommand{\sh}{\text{sh}}
\newcommand{\sm}{\text{sm}}
\newcommand{\tor}{\text{tor}}

\newcommand{\iy}{\infty}

\newcommand{\bk}{\bigskip}
\newcommand{\fc}{\frac}
\newcommand{\G}{\Gamma}

\newcommand{\s}{\sigma}

\newcommand{\om}{\omega}
\newcommand{\Om}{\Omega}

\newcommand{\BC}{\Bbb{C}}
\newcommand{\BG}{\Bbb{G}}
\newcommand{\BF}{\Bbb{F}}
\newcommand{\C}{\mathcal{C}}

\newcommand{\BQ}{\Bbb{Q}}

\newcommand{\BR}{\Bbb{R}}

\newcommand{\BZ}{\Bbb{Z}}
\newcommand{\BA}{\Bbb{A}}

\newcommand{\T}{\mathcal{T}}
\newcommand{\MS}{\mathcal{S}}

\renewcommand{\a}{\alpha}

\begin{document}

\date{}

\title{The Discriminant of an Algebraic Torus}

\author{Rony A. Bitan}

\thanks{A part of a Ph.D. Thesis supervised by B. Kunyavski\u\i\ at the Department of Mathematics, 
Bar-Ilan University, Ramat-Gan, Israel. Partially supported by the ISF center of excellency grant 1438/06.}

\address{ R. Bitan, Department of
Mathematics, Technion-Israel Institute of Technology -- Technion City, Haifa 32000, Israel}
\email{rony.bitan@gmail.com}

\maketitle

\baselineskip 20pt
\setcounter{equation}{0}
\pagestyle{plain}
\pagenumbering{arabic}

\begin{abstract}
For a torus $T$ defined over a global field $K$, 
we revisit an analytic class number formula obtained by Shyr in the 1970's 
as a generalization of Dirichlet's class number formula. 
We prove a local-global presentation of the quasi-discriminant of $T$, 
which enters into this formula, in terms of cocharacters of $T$. 
This presentation can serve as a more natural definition of this invariant.
\end{abstract}

\setcounter{page}{1}

%*****************
\section*{Introduction}
%*****************

The well-known class number formula expresses the class number of a number field $K$ 
in terms of other arithmetic invariants of $K$:
$$ h = \fc{ w |\Delta|^{1/2} \rho }{ 2^r (2\pi)^t R} $$  
where $w$ is the number of roots of unity in $K$, 
$\Delta$ is the discriminant of $K$, 
$\rho$ is the residue of the Dedekind Zeta-function of $K$ at $s=1$, 
$R$ is the regulator of $K$ and 
$r$ (resp. $t$) is the number of real (resp. pairs of complex) embeddings of $K$. 

\bk

In the early 1960's, T. Ono defined in \cite{Ono} analogues of these invariants for 
algebraic tori defined over both kinds of global fields, 
namely, the number field case (denoted by case~(N)) 
and the case of algebraic function field in one variable 
over a finite field of constants $\BF_q$ (denoted by case~(F)). 
One of these invariants is the \emph{quasi-discriminant}.  
As in the case of the discriminant of a global field, 
the quasi-discriminant of $T$ is the volume 
of the fundamental domain of the maximal compact subgroup of $T(A_K)/T(K)$   
-- where $A_K$ is the adele ring -- with respect to the Tamagawa measure. 
J. M. Shyr gave in \cite{Shyr1} a similar definition 
in the case of algebraic $\BQ$-tori. 
In this new construction, other arithmetic invariants are taken from Ono's definition. 
This led him to a relation 
which can be viewed as an ``analogue of the class number formula" for algebraic $\BQ$-tori. 
This relation can be generalized to tori defined over any global field $K$, as follows:
\begin{equation} \label{Shyr global invariant}
c_T^\Shyr := |\Delta_K|^{-d/2} C_\iy \prod\limits_\fp L_\fp(1,\chi_T) \cdot \om_\fp(T_\fp(\CO_\fp)) 
          = \fc{\rho_T \tau_T w_T}{h_T R_T} 
\end{equation}
where $C_\iy$ is an archimedean factor 
and $\Delta_K$ is the discriminant of $K$, $d=\dim T$. 
For any prime $\fp$ of $K$, $K_\fp$ is the complete localization of $K$ at $\fp$, 
$T_\fp = T \otimes K_\fp$, $T_\fp(\CO_\fp)$ is the maximal compact subgroup of $T_\fp(K_\fp)$, 
$L_\fp(s,\chi_T)$ is the local Artin $L$-function 
and $\om_\fp$ is some Haar measure on $T_\fp(K_\fp)$.  
$\rho_T$ is the residue of the global Artin $L$-function at $s=1$,  
$w_T$ is the cardinality of the torsion part of the group of units of $T$, 
$h_T$ is the class number of $T$, 
$R_T$ is the regulator of $T$ (equal to 1 in case(F)),
and $\tau_T$ is the Tamagawa number of $T$. 
This work is described in the fourth Section and on the Appendix. 
We call $c_T^\Shyr$ the \emph{Shyr invariant}. 
   
\bk

Locally, let $K$ be a henselian local field 
with a ring of integers $\CO_\fp$ and a finite residue field $k$. 
Let $T$ be an algebraic $K$-torus splitting 
over a finite Galois extension $L/K$ with Galois group $\G$. 
We investigate the invariant
$L_\fp(1, \chi_T) \cdot \om_\fp(T(\CO_\fp))$.    
Ono defined the group $T(\CO_\fp)$ using the dual $\G$-module, namely its group of characters. 
In order to measure it, we would like to describe it as a group of $\CO_\fp$-points of some integral model. 
In the first Section we exhibit the properties and relation between two integral models of $T$, 
namely, the standard integral model $X$ defined by V. E. Voskresenskii (which is of finite type), 
and the well-known N\'eron-Raynaud integral model $\T$, (which is locally of finite type). 
After applying a smoothening process to $X$, the identity components of both models coincide. 
Let $\Phi_T$ be the $k$-scheme of the group of components of the reduction of $\T$ modulo $\fp$. 
In the second Section we prove that:
$$ L_\fp(1, \chi_T) \cdot \om_\fp(T(\CO_\fp)) = |\Phi_T(k)_\tor|. $$
In the third Section, we use a construction of Kottwitz in \cite{Ko}, 
to prove an isomorphism of $k$-schemes
$ \Phi_T(k) \cong (X_\bullet(T)_I)^{\left<F\right>} $ 
where $X_\bullet(T)$ is the cocharacter group of $T$, $I$ is the inertia subgroup of $\G$ and 
$F$ is the Frobenuis automorphism generating $\G / I$. 
This isomorphism gives us another computation of the local component in Shyr's invariant:
$$ L_\fp(1,\chi_T) \cdot \om_\fp(T_\fp(\CO_\fp)) = |\ker\left(1-F|X_\bullet(T)_I \right)_\tor|. $$
Globally, together with the infinite place information, 
we prove in the fourth Section our main Theorem, 
which can be viewed as a local-global result, 
in spirit of the Artin-Hasse conductor-discriminant formula:
\begin{maintheorem*} 
$$ c_T^\Shyr = |\Delta_K|^{-d/2} C_\iy \prod\limits_\fp \left| \ker \left(1-F_\fp|X_\bullet(T_\fp)_{I_\fp} \right)_\tor \right| 
             = \fc{\rho_T \tau_T w_T }{ h_T R_T }. $$ 
\end{maintheorem*}

From this formula one can see that the Shyr invariant 
can be decomposed into the product of the ``arithmetic-geometric part" 
(related to the discriminant of the ground field) 
and the ``algebraic" part (reflecting the Galois action on the cocharacters).

\begin{ack*}
\emph{I thank B. Kunyavski\u\i\ for his guidance and support.
I also thank Q. Liu, T. Richarz and the referee for helpful comments.}
\end{ack*}

\bk

%----------------------------------
\section{Integral models of algebraic tori}
%----------------------------------

Let $K$ be any field. An \emph{algebraic torus} $T$ is an algebraic $K$-group 
such that $T \otimes_K L \cong \BG_{m,L}^d$ for some finite Galois extension $L/K$ 
where $\BG_m$ is the multiplicative group and $d$ is the dimension of $T$. 
The smallest among such extensions is called the \emph{splitting field} of $T$. 
We write $T \in \C(L/K)$. 
We denote by $ X^\bullet(T) = \Hom(T \otimes_K L,\BG_{m,L}) $ the group of characters of $T$. 
For any intermediate field $K \subseteq F \subseteq L$, $X^\bullet(T)_F$ 
is the sublattice of characters defined over $F$. 

\bk

Now let $K$ be a local field which is the complete localization of a global field 
with respect to a prime $\fp$. 
We denote by $\CO_\fp$ the ring of integers of $K$ and by $U_\fp = \CO_\fp^\times$ its subgroup of units. 
Let $k = \CO_\fp / \fp$ be the residue field with cardinality $q$. 
Let $T \in \C(L/K)$ and denote by $\G = Gal(L/K)$ the Galois group and by $I$ its inertia subgroup. 

\bk

An $\CO_\fp$-\emph{integral model} of $T$ is an $\CO_\fp$-scheme $M$ 
whose generic fiber $M \otimes_{\CO_\fp} K$ is isomorphic to $T$. 
The \emph{reduction modulo $\fp$} of $M$ 
is its special fiber $M_\fp = M \otimes_{\CO_\fp} k$. 
Its identity component $M^0$ is an open subscheme of $M$ 
which is the identity component of the special fiber, 
i.e. such that $M^0 \otimes k = (M \otimes k)^0$ (see \cite[p.154]{BLR}). 

\bk

T. Ono defined $T(\CO_\fp)$ as the maximal compact subgroup of $T(K)$ 
with respect to the $\fp$-adic topology. 
This was done using the group of characters (see \cite[2.1.3]{Ono}):
$$ T(\CO_\fp) = \{x \in T(K) \ : \ \chi(x) \in U_\fp \ \ \forall \chi \in X^\bullet(T)_K \}. $$
However, in order to measure this group with respect to some local measure, 
this description may not be enough. 
We would like to write $T(\CO_\fp)$ as the group of $\CO_\fp$-points of some group scheme. 
Toward this end, we consider two integral models of $T$.

%-----------------------------------------------------------
\subsection{The standard integral model} \label{standard model}
%-----------------------------------------------------------

The following construction is due to V. E. Voskresenskii 
and can be found in \cite[\S1]{Pop}. 
Notations are as above and suppose that $(L:K)=n$ and $\dim T = d$. 
Then $X^\bullet(T)$ is spanned as a $\BZ$-lattice by a basis $\{\chi_i\}_{i=1}^d$. 
The Hopf algebra $B = (L[X^\bullet(T)])^\G$ is the coordinate ring of $T$. 
The isomorphism $T \otimes_K L \cong \BG_m^d$ 
is equivalent to the isomorphism of $L$-algebras: $B \otimes_K L = LB \cong L[X^\bullet(T)]$. 
Let $\{\om_i\}_{i=1}^n$ be an integral basis of $L$ over $K$. 
Then 
$$ L[X^\bullet(T)] = B\om_1 \oplus ... \oplus B\om_n. $$ 
Thus there are linear combinations:
$$      \chi_i = x_i^{(1)}\om_1 + ... + x_i^{(n)}\om_n, \ \ x_i^{(j)} \in B $$
$$ \chi_i^{-1} = y_i^{(1)}\om_1 + ... + y_i^{(n)}\om_n, \ \ y_i^{(l)} \in B. $$

\begin{definition}
$A(X^\bullet(T)) = \CO_\fp[x_i^{(j)},y_m^{(l)}]$ is a Hopf $\CO_\fp$-algebra.
The group $\CO_\fp$-scheme $X = \Sp A(X^\bullet(T))$ 
is called the \emph{standard integral model} of $T$. 
\end{definition} 

\begin{remark} (\cite[\S~10.3]{Vos}) \label{properties of standard model}
Being obtained from a linear representation, 
$X$ is of finite type over $\CO_\fp$, and is reduced and faithfully flat. 
Further, $X(\CO_\fp) = T(\CO_\fp)$ is the maximal compact subgroup of $T(K)$ 
with respect to the $\fp$-adic topology. 
\end{remark}

%-------------------------------------------------------------
\subsection{The N\'eron-Raynaud integral model and relations to the standard model} \label{models properties}
%-------------------------------------------------------------

Let $K^\sh$ be the strict henselization of $K$ and $\CO_\fp^\sh$ be its ring of integers. 
We refer to the N\'eron-Raynaud (NR for short) model $\T$ of $T$ 
which is locally of finite type over $\CO_\fp^\sh$, 
as defined in \cite[ch.~10]{BLR} 
and satisfying $\T(\CO_\fp^\sh) = T(K^\sh)$.

\begin{remark} \label{lowering to local field}
A local ring is strictly henselian if its residue field is separably closed. 
In our cases (N) and (F) the residue field of any complete localization is perfect, 
thus $K^\sh$ is the maximal unramified extension of $K$ 
and it is a Galois over $K$. 
Let $\G_\sh = Gal(K^\sh/K)$. 
The $\G_\sh$-invariant subgroup of $\T(\CO_\fp^\sh) = T(K^\sh)$ is then $\T(\CO_\fp) = T(K)$. 
Indeed, as $\T$ is separated, 
the canonical map $\T(\CO_\fp) \to T(K)$ is injective.
It is also surjective by the universal property of the NR-model. 
\end{remark}

The following construction can be found in \cite{VKM}. 
Let $T_L = T \otimes L \cong \BG_{m,L}^d$. 
Let $\CO_\fP$ be the ring of integers of $L$ 
and let $\T_L$ be the NR-model of $T_L$ defined over it. 
The $\CO_\fp$-scheme $\MS = R_{\CO_\fP/\CO_\fp}(\T_L)$ 
obtained by the Weil restriction of scalars, is the NR-model of $R = R_{L/K}(T_L)$. 
Its identity component is $\MS^0 = R_{\CO_\fP / \CO_\fp}(\BG_{m,\CO_\fP}^d)$. 
Let $\N$ be the schematic closure of $T$ in $\MS$ 
induced by the canonical embedding $T \to R$. 
The standard $\CO_\fp$-model $X$ of $T$  
is isomorphic to $\N \cap \MS^0$ 
(see \cite[Prop.~6]{VKM}, 
the proof there is for $p$-adic fields 
but the arguments are valid also in case (F)).

\begin{lem} 
$X^0 = \N^0. $
\end{lem}
\begin{proof}
$ \N^0  = (\N \cap \MS^0)^0 = X^0. $
\end{proof}

The schemes $\N$ and $X$ are not necessarily smooth, 
i.e., their special fibers may not be reduced. 
To achieve the desired smooth NR-model $\T$, 
one should apply the \emph{smoothening process} (see \cite[Chap.~3]{BLR}). 
It is sufficient to control the defect of smoothness over $X = \N \cap \MS^0$ 
(\cite[Prop.~10.1/4]{BLR}). 
Thus the equality of the identity components of the two models is preserved. 
We denote the obtained \emph{smooth} standard model by $X_\sm$. 
As the ring representing $X$ is Notherian, this process consists of blowing up finitely many maximal ideals 
and $X_\sm$ remains of finite type. 
Moreover, by definition, the generic fibers of $X_\sm$ and $X$ are isomorphic.

\begin{cor} \label{equality of identity components}
$X_\sm^0 = \T^0$.
\end{cor} 

\bk

%------------------------------------------------
\section{Reductions and local volume computations}
%------------------------------------------------

%------------------------------------------------------ 
\subsection{Rational points of the group of components}
%------------------------------------------------------

Denote by $i: \Sp k \to \Sp \CO_\fp$ the canonical closed immersion of the special point. 
We call the $k$-scheme $\Phi_T = i^*(\T / \T^0)$ the \emph{group of components} of $\T$. 
There is an exact sequence of $k$-schemes:
$$ 1 \to \T^0_\fp \to \T_\fp \to \Phi_T \to 1, $$
where $\T_\fp = i^*(\T)$ and $\T_\fp^0 = i^*(\T^0)$. 
Let $l$ be the residue field of $L$ and let $\fg = Gal(l/k)$. 
Since $k$ is finite and $\T_\fp^0$ is affine and connected, 
by Lang's Theorem (see \cite[Ch.~VI,Prop.~5]{Ser}), 
$H^1(\fg,\T_\fp^0(l))$ is trivial implying the exactness of: 
\begin{equation} \label{k rational points of Phi}
1 \to \T_\fp^0(k) \to \T_\fp(k) \stackrel{\varphi}{\rightarrow} \Phi_T(k) \to 1.
\end{equation} 
Consider the map composition: 
$$ T(K) = \T(\CO_\fp) \stackrel{r}{\rightarrow} \T_\fp(k) 
                                      \stackrel{\varphi}{\rightarrow} \Phi_T(k) $$
where $r$ is the reduction modulo $\fp$ map and $\varphi$ is the map in (\ref{k rational points of Phi}). 
As $\T$ is smooth and $\CO_\fp$ is complete (and therefore henselian), 
$r$ is surjective (see \cite[Prop.~2.3/5]{BLR}). 
Since $\varphi$ is also surjective and the kernel of $r \circ \varphi$ 
is well-known to be $\T^0(\CO_\fp)$ (see, e.g., \cite[p.~1153]{Gon}), we obtain: 
\begin{lem} \label{NR-group of components}
$ T(K)/\T^0(\CO_\fp) \cong \T_\fp(k)/\T_\fp^0(k) = \Phi_T(k)$.
\end{lem}

The same construction for the smooth standard model $X_\sm$ with its group of components $\phi_T = i^*(X_\sm / X_\sm^0)$ 
leads to the corresponding isomorphism of abelian groups:  
\begin{equation} \label{Standard group of components}
X_\sm(\CO_\fp) / X_\sm^0(\CO_\fp) \cong \phi_T(k).
\end{equation} 

\begin{lem} \label{Phi_T(k)_tor}
$ X_\sm(\CO_\fp) / \T^0(\CO_\fp) \cong \phi_T(k) = \Phi_T(k)_\tor. $
\end{lem}
\begin{proof}
By Corollary \ref{equality of identity components}, $X_\sm^0(\CO_\fp) = \T^0(\CO_\fp)$. 
Recall that $\T$ is the smooth schematic closure $\N_\sm$ of $T$ in $\MS = R_{\CO_\fP/\CO_\fp}(\T_L)$, 
whereas $X_\sm = \N_\sm \cap \MS^0$ (see (\ref{models properties})). 
Thus as an abelian group: 
$$ \T(\CO_\fp) / X_\sm(\CO_\fp) \subseteq \MS(\CO_\fp) / \MS^0(\CO_\fp) \cong \Phi_R(k), $$
where $\Phi_R$ is the group of component of $R_{L/K}(\BG_m^d)$ and it is free (see \cite[Lem.~2.6]{Xar}). 
We get a decomposition of abelian groups: 
$$ \Phi_T(k) \cong T(K) / \T^0(\CO_\fp) \cong \T(\CO_\fp) / X_\sm(\CO_\fp) \times X_\sm(\CO_\fp) / \T^0(\CO_\fp) $$
on which the first factor is free and the second is finite (see Remark \ref{properties of standard model}) 
and therefore is the torsion part of $\Phi_T(k)$. 
\end{proof}

%-------------------------------------
\subsection{local volume computations}
%-------------------------------------

As $K$ is locally compact, it admits a left invariant Haar measure. 
We normalize such an additive measure $dx$ on $K$ by requiring: 
$dx(\fp)=q^{-1}$, which is equivalent to $dx(\CO_\fp)=1$. 
This induces a multiplicative Haar measure $\om_\fp$ 
on the group of points $T(K)$ (see \cite[\S~2.2]{Weil}). 

\begin{definition}
Let $M$ be one of the aforementioned $\CO_\fp$-models of a $K$-torus $T$, 
namely the (smooth) standard model or the NR one. 
We call the reduction of $M$ ``\emph{good}" if $M^0_\fp$ is a $k$-torus. 
As the identity components of these two models coincide, 
the definition of good reduction does not depend on 
the choice of a model. 
\end{definition}

\begin{prop} \emph{(\cite[Prop.~1.1]{NX})} 
A $K$-torus $T$ has good reduction if and only if it splits over an unramified extension. 
This means that $I$ acts trivially on $X^\bullet(T)$. 
\end{prop}

\begin{remark} (\cite[Prop.~1.2]{NX}) 
Let $Y$ and $N$ be the kernel and image of the map
$$ tr: X^\bullet(T) \to X^\bullet(T)^I, \ \chi \mapsto \sum\limits_{\s \in I} \chi^\s. $$
Then the exact sequence: \label{canonical modules decomposition}
$$ 0 \to Y \to X^\bullet(T) \to N \to 0 $$
induces an exact sequence of $K$-tori:
\begin{equation} \label{canonical tori decomposition}
1 \to T_I \to T \to T_a \to 1
\end{equation}
on which $T_I$ is the maximal subtorus of $T$ having good reduction 
whereas $T_a$ is \emph{$I$-anisotropic}, i.e. $X^\bullet(T_a)^I = \{0\}$, 
standing extremely on the other edge with regard to the good reduction one. 
Sequence (\ref{canonical tori decomposition}) shows that any torus is an extension of such two pieces. 
\end{remark}

\begin{remark} \cite[Thm.~1.3]{NX} \label{toric part}
Let $T_{(\fp)}$ be the toric part of $\T_\fp^0$. There is an isomorphism of $\G/I$-modules:
$$ X^\bullet(T_{(\fp)}) \cong X^\bullet(T_I) \cong X^\bullet(T) / \ker (X^\bullet(T) \stackrel{tr}{\rightarrow} X^\bullet(T)^I). $$ 
\end{remark}

\begin{remark} \label{bad torus}
A $K$-torus $T$ admits a finite type NR-model $\T$ if and only if $T \otimes_K K^\sh$ 
does not contain any subgroup of type $\BG_m$ (see \cite[10.2.1]{BLR}), i.e., it is $I$-anisotropic. 
In this case $\T$ coincides with the smooth standard model $X_\sm$ (see \cite[Prop.~10.8]{Pop}). 
\end{remark}

We briefly introduce now the local Artin functions which serve as a system of convergence factors 
in the infinite product of local measures on the global Shyr invariant (see formula (\ref{Shyr global invariant})). 
The following definitions can be found in \cite[\S~13]{Vos} and in \cite[Chap.~VII,\S~10.1]{Neu}. 
The Galois group of the maximal unramified subextension in $L/K$, namely $\G / I$, 
is isomorphic to $\fg = Gal(l/k)$ 
and is generated by the Frobenius automorphism $F$. 
$\fg$ acts naturally on $X^\bullet(T)^I$, 
inducing an integral representation: 
$$ h : \fg \rightarrow Aut(X^\bullet(T)^I) \cong \textbf{GL}_{d_I}(\BZ), 
           \ d_I = \text{rank} (X^\bullet(T)^I). $$ 
Denote the character of this representation by $\chi_T$. 

\begin{defn} \label{local Artin L-function}
The \emph{local Artin L-function} for $T$ is defined by
$$ L_\fp(s,\chi_T,L/K) = L_\fp(s,\chi_T) = \det \left( 1_d - \fc{h(F)}{q^s} \right)^{-1}, $$
where $s \in \BC$ with $Re(s) > 1$. 
\end{defn}

\begin{theorem} \emph{\cite[Thm.~14.3/3]{Vos}} \label{good reduction measure} \\
If $T$ splits over an unramified extension then: $|T(k)| \cdot q^{-d} = L_\fp(1,\chi_T)^{-1}, \ d=\dim T$. 
\end{theorem}

\begin{defn} \label{isogeny definition} 
Let $G_1,G_2$ be algebraic groups defined over a field $K$. 
An \emph{isogeny} $\lambda: G_1 \to G_2$ is a surjective homomorphism of algebraic groups with finite kernel. 
We denote it by $\lambda: G_1 \rightleftharpoons G_2$. 
\end{defn}

\begin{theorem} \label{1-cohomology vanishes on free group of component} \emph{(\cite[2.19]{Xar})} 
Let $T$ be an algebraic torus defined over a local field $K$ 
splitting over a Galois extension $L$ with inertia subgroup $I$. 
%Then  
%a. 
$\Phi_T$ is torsion-free if and only if $H^1(I,X^\bullet(T))=0.$ 
%b. $\Phi_T$ is finite if and only if $X^\bullet(T)^I = 0$.
\end{theorem}

As one can observe from Theorem \ref{1-cohomology vanishes on free group of component}, 
the abelian group $H^1(I,X^\bullet(T))$ which is isomorphic by Tate-duality to $H^1(I,T)$, 
measures the lack of connectedness of the torsion part in $\Phi_T(k)$. 
As this property is not invariant under isogeny, 
according to Corollary \ref{Phi_T(k)_tor}, 
this prevents at the same time the Shyr invariant to be respected by an isogeny of tori.

\begin{prop} \label{rho invariant fixed under isogeny} \emph{(\cite[\S1.3.1]{Ono})}
Let $K$ be any field and $L$ be a finite Galois extension with $\G = Gal(L/K)$. 
Let $T^*,T \in \C(L/K)$. Then: 
$$ T^* \rightleftharpoons T \ \ \Leftrightarrow \ \ X^\bullet(T^*) \otimes \BQ \cong X^\bullet(T) \otimes \BQ \ \ \text{as $\G$-modules}. $$   
\end{prop}

\begin{lem} \label{toric parts isogeny}
Let $T^*,T$ be $K$-isogenous tori.
Then the toric parts of their reductions are $k$-isogenous. 
\end{lem}

\begin{proof}
Consider the two exact sequence of $\G$-modules for $T$ and $T^*$, 
induced by the canonical decomposition (\ref{canonical tori decomposition}). 
Since $\BQ$ is flat as a $\BZ$-module, exactness is preserved after tensoring with it:
$$ 0 \to X^\bullet(T_a)   \otimes \BQ \to X^\bullet(T)   \otimes \BQ \to X^\bullet(T_I) \otimes \BQ \to 0 $$
$$ 0 \to X^\bullet(T_a^*) \otimes \BQ \to X^\bullet(T^*) \otimes \BQ \to X^\bullet(T^*_I) \otimes \BQ \to 0. $$ 
Recall from Theorem \ref{toric part} that: 
$$ X^\bullet(T_I) \cong X^\bullet(T) / \ker (X^\bullet(T) \stackrel{tr}{\rightarrow} X^\bullet(T)^I). $$ 
As $T$ and $T^*$ are $K$-isogenous, 
according to Theorem \ref{rho invariant fixed under isogeny}, 
$ X^\bullet(T) \otimes \BQ \cong X^\bullet(T^*) \otimes \BQ$ and therefore:
$$ \ker (X^\bullet(T) \otimes \BQ \stackrel{tr}{\rightarrow} X^\bullet(T)^I \otimes \BQ) 
 = \ker (X^\bullet(T^*) \otimes \BQ \stackrel{tr}{\rightarrow} X^\bullet(T^*)^I \otimes \BQ) $$
implying that
$ X^\bullet(T_I) \otimes \BQ \cong X^\bullet(T^*_I) \otimes \BQ. $
By Remark (\ref{toric part}) 
this is equivalent to an isomorphism of their reduction toric part as $\fg = Gal(l/k)$-modules:
$$ X^\bullet(T^*_{(\fp)}) \otimes \BQ \cong X^\bullet(T_{(\fp)}) \otimes \BQ, $$
which again by Proposition \ref{rho invariant fixed under isogeny} 
implies that $T^*_{(\fp)}$ and $T_{(\fp)}$ are $k$-isogenous. 
%This also implies that they share the same Artin $L$-functions (see \cite[p.~106]{Vos}).
\end{proof}

We consider another good reduction torus associated to $T$,  
namely, the factor torus $T^I$ corresponding to the $\G$-module $X^\bullet(T)^I$. 

\begin{lem} \label{T_I and T^I}
$T_I$ and $T^I$ are $K$-isogenous.
\end{lem}

\begin{proof}
Consider again the canonical decomposition of $\G$-modules (\ref{canonical modules decomposition}):
$$ 0 \to X^\bullet(T_a) \to X^\bullet(T) \stackrel{tr}{\rightarrow} X^\bullet(T_I) \to 0. $$
Taking the $I$-invariants gives the long exact sequence:
$$ 0 \to X^\bullet(T_a)^I = 0 \to X^\bullet(T)^I \to X^\bullet(T_I)^I = X^\bullet(T_I) \to H^1(I,T_a). $$
The finiteness of $H^1(I,T_a)$ implies the one of $\cok(X^\bullet(T)^I \to X^\bullet(T_I))$, 
which means that the $\BZ$-lattice $X^\bullet(T)^I$ 
is a sublattice of finite index in the $\BZ$-lattice $X^\bullet(T_I)$. 
Back to $K$-tori, this indicates that the corresponding epimorphism $T_I \to T^I$ has a finite kernel. 
\end{proof}

\begin{prop} \label{measure of identity component}
$ \ \ \om_\fp(\T^0(\CO_\fp)) = |\T_\fp^0(k)| \cdot q^{-d} = L_\fp(1,\chi_T)^{-1}. $ 
\end{prop}

\begin{proof} 
$\T^0$ is the reduction preimage of the $k$-group $\T_\fp^0$, 
thus it is smooth and therefore the reduction of points $\T^0(\CO_\fp) \to \T_\fp^0(k)$ is surjective. 
Consider the exact sequence: 
\begin{equation} \label{identity component points decomposition}
1 \to \T^1(\CO_\fp) \to \T^0(\CO_\fp) \to \T_\fp^0(k) \to 1.
\end{equation}
%Then:
%$$ \T^1(\CO_\fp) \cong (\T^1(\CO_\fP))^\G = (1+\fp)^d $$
The reduction image of $\T^1(\CO_\fp)$ is the $d$-tuple $(1,..,1)$ in $\T_\fp^0(k)$ where $d = \dim \T^0 = \dim T$. 
It is isomorphic to another preimage of this map, namely $(1+\fp)^d$, 
which is homeomorphic to the additive group $\fp^d$    
implying that $\om_\fp(\T^1(\CO_\fp)) = \bigwedge_{i=1}^d dx_i(\fp^d) = q^{-d}$
and consequently by (\ref{identity component points decomposition}): 
$$ \om_\fp(\T^0(\CO_\fp)) = |\T_\fp^0(k)|  \cdot q^{-d}. $$
As for the right hand equality of the Proposition, 
$\T_\fp^0$ is an affine smooth and connected $k$-group (see \cite[\S1]{NX}). 
It has a canonical decomposition over $k$: 
$$ \T_\fp^0 = T_{(\fp)} \times U $$  
where $T_{(\fp)}$ is a $k$-torus and $U$ is a unipotent $k$-group. 
$U$ is isomorphic to an affine space $\BA_k^{\dim U}$ thus $ |U(k)| = q^{\dim U}$ 
and therefore: $|\T_\fp^0(k)| \cdot q^{-d} = |T_{(\fp)}(k)| \cdot q^{-d_I}$ 
where $d_I = \dim T_{(\fp)}$. 
Let $T_I$ be the maximal subtorus of $T$ with good reduction.  
From Remark \ref{toric part}, $X_\bullet(T_I) \cong X_\bullet(T_{(\fp)})$ as $\G/I$-modules. 
Thus we may deduce that $T_{(\fp)}$ is the reduction of $T_I$, splitting over an unramified extension. 
Hence by Theorem \ref{good reduction measure} $|T_{(\fp)}(k)| \cdot q^{-d_I} = L_\fp(1,\chi_{T_I})^{-1}.$ 
According to Lemma \ref{T_I and T^I}, $T_I$ and $T^I$ are $K$-isogenous. 
These tori have good reduction, thus due to Lemma \ref{toric parts isogeny} 
their reductions (being $k$-tori) are $k$-isogenous,  
implying that their $L$-functions coincide (see \cite[p.~106]{Vos}). 
In particular $L_\fp(1,\chi_{T_I})$ is equal to $L_\fp(1,\chi_{T^I})$
which is by definition equal to $L_\fp(1,\chi_T)$. 
\end{proof}

As noted in Remark \ref{properties of standard model}, 
$X(\CO_\fp)$ is the maximal compact subgroup of $T(K)$.  
Further, it is equal to $X_\sm(\CO_\fp)$ (see \cite[\S3.1~Def.~1]{BLR}). 
From Lemmas \ref{Phi_T(k)_tor} and \ref{measure of identity component} 
the local component in the Shyr invariant can be computed by:
\begin{cor} \label{local Shyr invariant computation}
$ L_\fp(1,\chi_T) \cdot \om_\fp(X_\sm(\CO_\fp)) = (X_\sm(\CO_\fp) : \T^0(\CO_\fp)) = |\Phi_T(k)_\tor|. $
\end{cor}

\bk

%*********************************************
\section{Relation with the cocharacter group}
%*********************************************

The following construction can be found in \cite[\S7.2]{Ko}.  
It was originally defined over a $p$-adic field 
but as we shall see, it can be applied also in case (F).
Let $K^\sh$ be the strict henselization 
of the local field $K$ in a separable closure $K_s$. 
As $K^\sh$ is the maximal unramified extension of $K$, 
the group $Gal(K_s / K^\sh)$ is the inertia subgroup $I$ of the absolute one $Gal(K_s / K)$. 
R. Kottwitz extends the canonical epimorphism
$ (K^\sh)^\times \to \BZ $
with kernel equal to the group of units of $K^\sh$, to an epimorphism
$$ \mathcal{K}_T:T(K^\sh) \to X_\bullet(T)_I $$
where the latter group is the $I$-coinvariants of the cocharacter group. 
Let $\T$ be the NR-model of $T$ defined over the ring of integers $\CO^\sh_\fp$ of $K^\sh$. 

\begin{lem}
$ \ker(\mathcal{K}_T) = \T^0(\CO^\sh_\fp)$. 
\end{lem}

\begin{proof}
As noted in the first line of the proof of \cite[Prop.~3,~p.189]{HR}, 
$\ker(\mathcal{K}_T)$ is the unique Iwahori subgroup of $T(K^\sh)$. 
Thus, by \cite{HR}, Definition 1, p.188, and the statement of the cited proposition, 
$\ker(\mathcal{K}_T)$ coincides with $\T^0(\CO_\fp^\sh)$ (see \cite[Remarks~2.2.(iii)]{RP}). 
Note that the proof applies to any strictly henselian discretely-valued field 
and therefore covers both cases (N) and (F).
\end{proof}
%We briefly repeat the definition of an Iwahoric subgroup, 
%as defined by M. Rapoport in \cite[Chap.~I/2]{RP},  
%in which the discussion is more generally 
%on the Kottwits epimorphism $\mathcal{K}_G$ 
%for a connected reductive group $G$ defined over $K^\sh$.  
%Let $\CG$ be the NR-model of $G$ defined over $\CO_\fp^\sh$ 
%and let $\CG_\fp^0$ be the identity component of its reduction. 
%Let $\ov{B}$ be the Borel subgroup of $\CG_\fp^0(k_s)$. 
%Consider the reduction of points:
%$$ r:G(K^\sh) = \CG(\CO_\fp^\sh) \to \CG_\fp(k_s). $$
%Then $B = r^{-1}(\ov{B})$ is an \emph{Iwahoric subgroup} of $G(K^\sh)$.  
%Let $S$ be the maximal torus in $B$. 
%We refer to the Bruhat-Tits building $\B$ associated to the adjoint group of $G$, on which 
%the apartments correspond to conjugations of $S$ in $G(K^\sh)$ and 
%the chambers, being the the maximal facets, correspond to Iwahoric subgroups, 
%i.e. to conjugations of $B$. 
%Any facet $F$ of $\B$ defines a \emph{parahoric subgroup} of $G(K^\sh)$, namely: 
%$$ K_F = \text{Fix(F)} \cap \ker(\mathcal{K}_G). $$
%In our case of a $K^\sh$-torus $T$ with NR-model $\T$, 
%$\ov{B} = \T_\fp^0(k)$ and $B = \T^0(\CO_\fp^\sh)$. 
%Moreover, as $T$ is commutative, 
%$B=S$ is its unique conjugation in $T(K^\sh)$, 
%i.e., it is the unique Iwahoric subgroup and it is equal to $\ker(\mathcal{K}_T)$, 

Since the residue field of $K^\sh$ is $k_s$, 
the group of components of $\T$ splits over it, i.e. $\Phi_T(k_s) = \Phi_T$. 
Hence together with Lemma \ref{NR-group of components}, 
the Kottwitz construction gives rise to an isomorphism

\begin{equation} \label{Phi and Kottwitz}
T(K^\sh) / \T^0(\CO_\fp^\sh) \cong \Phi_T \cong X_\bullet(T)_I.
\end{equation}

Now let $T$ be defined over $K$ and let $\T$ be its NR-model. 
The absolute group $\fg_\sep = Gal(k_s / k)$ 
being generated by the Frobenuis automorphism $F$, 
is identified with $\G_\sh = Gal(K^\sh / K)$. 
%%%%%%%%%%%%%%%%%%%%%%%%%%%%%%%%%%%%%%%%%%%%%%%%%%%%%
%\begin{lem}
%Let $\G$ be an affine smooth group scheme over $\CO_\fp$ 
%with geometrically connected fibers. 
%Then $H^1(\left< F \right>,\G(\CO_\fp^\sh))$ is trivial.
%\end{lem}
%%%%%%%%%%%%%%%%%%%%%%%%%%%%%%%%%%%%%%%%%%%%%%%%%%%%%
The scheme $\T^0$ has a geometrically connected fiber. 
Moreover, it is affine over $\CO_\fp$ (see \cite[Prop.~3]{KM} and \cite[p.~290]{BLR}). 
Thus by Lang's Theorem, the cohomology group 
$$ H^1(\left<F\right>,\T^0(\CO_\fp^\sh)) = H^1(\left<F\right>,\T^0(k_s)) $$
where $k$ is considered as a $\CO_\fp$-algebra, is trivial.  
Hence taking the $\G_\sh$-invariants of (\ref{Phi and Kottwitz}) 
gives rise to an epimorphism (see \cite[7.6.2]{Ko}): 
$$ T(K) \to (X_\bullet(T)_I)^{\left<F\right>} $$
with kernel equals to $\T^0(\CO_\fp^\sh) \cap T(K) = \T^0(\CO_\fp)$. 
Again by Lemma \ref{NR-group of components} and Corollary \ref{local Shyr invariant computation} 
we get:

\begin{cor} \label{local Shyr invariant using cocharacter group}
$$ T(K) / \T^0(\CO_\fp) \cong \Phi_T(k) \cong \ker(1-F|X_\bullet(T)_I). $$
and:
$$ X_\sm(\CO_\fp) / \T^0(\CO_\fp) \cong \Phi_T(k)_\tor \cong \ker(1-F|X_\bullet(T)_I)_\tor. $$
\end{cor}

\begin{example} \label{example}
Let $L$ be a cyclic extension of $K$ with Galois group $\G = \left< \s \right>$. 
Let $R = R_{L/K}(\BG_m)$ be the corresponding Weil torus, 
i.e., such that for any $K$-algebra $B$, $R(B) = (B \otimes L)^\times$. 
The \emph{norm torus} $T' = R_{L/K}^{(1)}(\BG_m)$ is the kernel of the norm map $N_{L/K}:R \to \BG_{m,K}$, 
mapping any element of $R(B)$ to the product of its images under all Galois automorphisms. 
Suppose $L/K$ is totally ramified, i.e., $I=\G$. 
Then $T'$ is an $I$-anisotropic torus. 
Its NR-model $\T'$ which is of finite type, 
coincides with the smooth standard one $X'_\sm$ (see Remark \ref{bad torus}). 
Note that the character group $X^\bullet(R) = \BZ[\G]$ (see \cite[\S3.12]{Vos}), 
coincides as a $\G$-module with the group of cocharacters: 
$$ X_\bullet(R) = \Hom(\BG_m,R \otimes L) 
                = \Hom(\BZ,\BZ[\G]) = \BZ[\G] = X^\bullet(R). $$ 
Since $\G$ is cyclic, $T$ is isomorphic as a $\G$-module to the projective group $P = R / \BG_m$ 
(see \cite[p.~22]{LL}). 
The exact sequences of $K$-tori:
$$ 1 \to T' \to R \to \BG_m \to 1 $$
$$ 1 \to \BG_m \to R \to P \to 1  $$
induce the exact sequences of dual modules:
$$ 0 \to \BZ \to X^\bullet(R) \to X^\bullet(T') \to 0 $$
$$ 0 \to \BZ \to X_\bullet(R) \to X_\bullet(P) \to 0. $$
We get an isomorphism of $\G$-modules: $ X_\bullet(T') \cong X_\bullet(P) \cong X^\bullet(T'). $
Explicitly we have
$$ X'_\sm(\CO_\fp) / \T'^0(\CO_\fp) \cong X_\bullet(T')_I = \left( \BZ[\s] / \sum\limits_i \s^i \right)/ (1-\s) 
                                                    %= \BZ[\s]/(1-\s^n) 
                                                    = \mu_n $$ 
thus by Corollary \ref{local Shyr invariant computation} we get 
$L_\fp(1,\chi_{T'}) \cdot \om_\fp(X'_\sm(\CO_\fp)) = n.$ 
More generally, for any extension $L/K$, we have $X'_\sm(\CO_\fp) / \T'^0(\CO_\fp) \cong \phi_{T'} = \mu_e$ 
where $e$ is the ramification index (see \cite[Thm.~3]{Pop}).
\end{example}

\bk

%%%%%%%%%%%%%%%%%%%%%%%
\section{Ono and Shyr invariants}
%%%%%%%%%%%%%%%%%%%%%%%

In the following Section we briefly describe the analogues 
of the arithmetic invariants of number fields for algebraic tori as defined by Ono, 
and the analogue of the classical class number formula for algebraic $\BQ$-tori, 
as formulated by Ono and Shyr. 
This construction is generalized to $K$-tori 
where $K$ is any global field. 
Finally, our local results will be inserted in these invariant formulas.

%-----------------------------------
\subsection{Arithmetical invariants of algebraic tori}
%-----------------------------------

\begin{notation}
Let $K$ be a global field, i.e., either a number field 
or an algebraic function field in one variable over a finite field of constants $\BF_q$. 
We denote: \\
$\Delta_K $ -- the discriminant of $K$. In case~(F) equals to $q^{2g-2}$ where $g$ is the genus of $K$. \\
$S $ -- a finite set of valuations of $K$ which contains the set $S_\iy$ of the archimedean ones. \\  
$K_v $ -- the completion of $K$ with respect to a valuation $v \in S$. \\
$\CO_v $ -- the ring of integers of $K_v$ and $U_v = \CO_v^\times$ -- its subgroup of units. \\
$T_v = T \otimes K_v $ and 
$T_v(\CO_v) = \{ x \in T_v(K_v) : \chi(x) \in U_v \ \forall \chi \in X^\bullet(T)_{K_v} \}$. \\
If $v = \fp$ is a prime, $T_\fp(\CO_\fp)$ is the maximal compact subgroup of $T_\fp(K_\fp)$, \\ 
$\fP$ is a prime of $L$ lying over $\fp$,
$ \G_\fp = Gal(L_\fP/K_\fp) $ and $I_\fp$ is the inertia subgroup. 
\end{notation}

In Definition \ref{local Artin L-function} we defined the local Artin $L$-function. 
Globally, consider the action of $\G$ on $X^\bullet(T)$ and the corresponding representation 
$\G \to Aut(X^\bullet(T)) \cong \textbf{GL}_n(\mathbb \BZ)$. 
Its character $\chi_T$ 
is decomposed into a sum of irreducible characters of $\G$ with integral coefficients:
$$ \chi_T = \sum_{i=1}^m a_i \chi_i, \ a_i \in \BZ $$ 
where $\chi_1$ is the \emph{principal character}. \\
The \emph{global Artin L-function} is defined by the Euler product: 
$$ L(s,\chi_T) = L(s,\chi_T,L/K) = \prod_\fp L_\fp(s,\chi_{T_\fp},L_\fP / K_\fp) $$
again with $Re(s) > 1$, having a pole at $s=1$ of order $a_1$.

\begin{definition} 
The \emph{quasi-residue} of $T$ is the limit:
$$ \rho_T = \lim_{s \to 1} (s-1)^{a_1}L(s,\chi_T). $$
\end{definition}

Following Ono in \cite{Ono}, 
for any finite set of places $S$ which contains $S_\iy$ 
we define $T_A(S) = \prod_{v \in S} T_v \prod_{v \notin S} T_v(\CO_v)$. 
Then the \emph{adelic group} $T_A$ is the inductive limit of $T_A(S)$ with respect to $S$. 
The \emph{group of $S$-units} is $T_K(S) = T_K \cap T_A(S)$. 
The \emph{group of units} of $T$ is then $T_K(S_\iy)$ where $S_\iy$ is the set of archimedean places, 
which are the elements of $\G$ composed with the absolute norm $|\cdot|_\iy$.

\begin{theorem} \label{Shyr's generalization of Dirichlet Unit Theorem} 
Shyr's generalization of Dirichlet Unit Theorem - see \cite{Shyr2} \\
The group $T_K(S)$ is a direct product of the finite group $T_K \cap T_A^c$ where: $T_A^c = \prod_v T_v(\CO_v)$ 
and a group isomorphic to $\BZ^{r(S)-r}$ where: $r(S) = \sum_{v \in S} r_v$. 
\end{theorem}

Let $T$ be a $\BQ$-torus and let $r_{\BQ} = rank(X^\bullet(T)_\BQ), \ r_\iy = rank(X^\bullet(T)_\BR)$. 
Let $\{\xi_i\}_{i=1}^{r_\iy}$  be a $\BZ$-basis of $X^\bullet(T)_\BR$ 
such that $\{\xi_i\}_{i=1}^{r_{\BQ}}$ is a $\BZ$-basis of $X^\bullet(T)_\BQ$. 
Then the group of units $T_K(S_{\infty})$ is decomposed into $W \times E$ 
where $W$ is finite and $E \cong \BZ^{r_\iy-r_\BQ}$. 
We denote $w_T = |W|.$

\begin{definition} 
Let $\{ e_j \}_{j=r_\BQ+1}^{r_\iy}$ be a $\BZ$-basis of $E$. 
The number $R_T = |\det(\ln|\xi_i(e_j)|_{r_\BQ+1 \leq j \leq r_{\iy}})|$ 
is called the \emph{regulator} of $T$ over $\BQ$. 
Geometrically, this number represents (as for number fields) the volume of the fundamental domain 
for the free part of group of units. 
\end{definition}

We set:
$ T_A^1 = \left \{ x \in T_A \ : \ \chi(x) \in I_K^1  \ \ \forall \chi \in X^\bullet(T)_K \right \} $
where $I_K^1 = \{a \in I_K \ : \ ||a||=1 \}$ and $I_K$ denotes the idele group. 
It is the maximal subgroup of $T(A_K)$ such that $T_A^1 / T(K)$ is compact. 

\begin{definition}
The \emph{class number} of $T$ is the finite index:

\begin{equation*}
h_T = 
\left \{ \begin{array}{l l}
\, [T_A:T_K \cdot T_A^{S_\iy}] \ \ \text{case(N)} \\ 
\\
\, [T_A^1:T_K \cdot T_A^{S_\iy}] \ \ \text{case(F)} \\
\end{array}\right.
\end{equation*}

\end{definition}

%Since the adele group $T_A$ is locally compact in the adele topology, 
%it admits a left invariant Haar measure which is determined uniquely up to a positive factor. 
%Let $dx_v$ be the normalized Haar measure on $K_v$ (see \cite[\S~14]{Vos}). 
%This measure induces a multiplicative Haar measure on $T(K_v)$ (see \cite[\S~2.2.b]{Weil}).
The global measure is obtained by the infinite product of these local measures, 
multiplied by a set of convergence local factors, 
namely, the Artin local $L$-functions $L_\fp(1,\chi_T)$: 
$$ \tau = \rho_T^{-1}|\Delta_K|^{-\fc{d}{2}} \prod_{v|\iy} \om_v \prod_\fp L_\fp(1,\chi_T)\om_\fp $$
This is called the \emph{Tamagawa measure}. 
This measure applied to $\prod_\fp T_\fp(\CO_\fp)$ is convergent --
almost all primes are unramified on which $L_\fp(1,\chi_T) \cdot \om_\fp(T_\fp(\CO_\fp)) = 1$.

%--------------------------------------
\subsection{The quasi-discriminant definition}
%--------------------------------------

The geometrical meaning of a global field discriminant 
is the volume of a fundamental domain of its ring of integers. 
An analogue for the case of an algebraic torus $T$ defined over a global field, 
called the ``\emph{quasi-discriminant}" of $T$, 
is the volume -- with respect to some normalized Haar measure -- 
of the fundamental domain of the maximal compact subgroup of $T(A_K)/T(K)$. 
In the next Section we refer to two similar analogues, 
given by T.~Ono and J.~M.~Shyr, to this invariant in the case of algebraic tori.

%------------------------------
\subsubsection{Ono's invariant}
%------------------------------

The following construction can be found in \cite{Ono}. 
Let $\{ \chi_i \}_{i=1}^{r_K}$ be a $\BZ$-basis of $X^\bullet(T)_K$. 
Consider the epimorphism $\psi:T_A \rightarrow \BR_+^{r_K}$ defined by 
$$ \a \mapsto (\ln|\chi_i(\a)|)_{1 \leq i \leq r_K}. $$
It yields the isomorphism $T_A/T_A^1 \cong \BR_+^{r_K}$. 
For a Haar measure $d$ defined on $T_A$, consider the decomposition
$$ d(T_A/T_K) = d(T_A/T_A^1) \cdot d(T_A^1/T_K). $$ 
Let $t_K$ be the pullback of the Haar measure $\frac{dx_1 \cdot \cdot \cdot dx_r}{x_1 \cdot \cdot \cdot x_r}$  
on $\BR_+^{r_K}$ from $T_A/T_A^1$.
Define the \emph{normalized} Haar measure $\Omega_T$ on $T_A$, 
i.e., such that $\int_{T_A^1/T_K} d(\Omega_T/t_K) =1$.

\begin{definition} 
Comparing the measure $\Omega_T/t_K$ with the Tamagawa measure gives the constant 
$$ c_T^\Ono = \frac{\om_T}{\Omega_T/d_K} $$ 
%which will be called \emph{Ono's quasi-discriminant} of $T$. 
\end{definition}

Now assume $K=\BQ$. Let $\{ \xi_i \}_{i=1}^{r_\iy}$ be a $\BZ$-basis of $X^\bullet(T)_\BR$ 
such that $\{ \xi_i \}_{i=1}^{r_\BQ}$ is a $\BZ$-basis of $X^\bullet(T)_\BQ$. 
Define $ \Phi_0 : T_\BR \rightarrow \BR^{r_\iy}$ by 
$$ x \mapsto (\ln|\xi_i(x)|)_{1 \leq i \leq r_{\iy}}.$$ 
By the unit Theorem, $\text{rank}(U) = r_\iy - r$. 
Let $\{e_j\}_{j=r_\BQ+1}^{r_\iy}$ be a $\BZ$-basis of $E$ and consider the parallelotope:
$$ P_0 = \left \{ \sum_{j=r_\BQ+1}^{r_\iy} \lambda_j \Phi_0(e_j) \ : \ 0 \leq \lambda_j \leq 1 \right \}.$$ 
This is the fundamental domain of $E$ on $\BR^{r_\iy-r_\BQ}$ 
and its Euclidean volume is the regulator $R_T$.  
This domain may be extended to dimension $r_\iy$ 
by the cube $(1 \leq \lambda_j \leq e)_{1 \leq j \leq r_\BQ}$. 
The extended embedding $\Phi: T_A^{S_\iy} \to \BR^{r_\iy}$ defined by $\Phi(x) = \Phi_0(x_\iy)$ 
(ignoring the non-archimedean places components), 
gives rise to a parallelotope $P$ by a continuation of a unit cube on $\BR^{r_\iy}$. 
Thus the Euclidean volume of $P$ on $\BR^{r_\iy}$ remains $R_T$.

%\begin{lem} \emph{(\cite[3.8.3]{Ono})} 
%Let $T_A^{1,\iy} = T_A^1 \cap T_A^{S_\iy}$. 
%Then $T_A^{1,\iy} = T_\BZ \cdot \Phi^{-1}(P_0)$ 
%and: $\Phi^{-1}(P_0) \cap T_\BQ = W.$
%\end{lem}

\begin{lem} \emph{(\cite[3.8.5]{Ono})}
$ \int_{\Phi^{-1}(P)} d(\Omega_T/t_\BQ) = w_T / h_T. $
\end{lem}

Denote by $I$ the unit cube in $\BR^{r_\iy}$. 
Define by $M_\iy$ the $\iy$ component of $\Phi^{-1}(I)$, i.e.,
$$ M_\iy = \{ x \in T_\BR \ : \ 1 \leq |\xi_i(x)|_\iy \leq e, \ 1 \leq i \leq r_\iy \}. $$
The regulator, being the Euclidean volume of $P$ on $\BR^{r_\iy}$, is obtained by
$$ \fc{ \int_{\Phi^{-1}(P)} d(\Om_T/t_\BQ) }{ \int_{\Phi^{-1}(I)} d(\Om_T/t_\BQ) } $$ 
hence we get
\begin{align*} 
   c_T^{Ono} & = \fc{\om_T}{\Om_T/t_\BQ} (\Phi^{-1}(P)) = \fc{R_T \om_T(\Phi^{-1}(I))} {\Om_T/t_{\BQ}(\Phi^{-1}(P))} 
               = \fc{R_T \int_{M_\iy} \om_\iy \prod _\fp L_\fp(1,\chi_T)\om_\fp }{w_T/h_T} \\
             & = \fc{R_T h_T}{w_T} \int_{M_\iy} \om_\iy \prod _\fp \int_{T_\fp(\CO_\fp)} L_\fp(1,\chi_T)\om_\fp. 
\end{align*}

According to Theorem \ref{Shyr's generalization of Dirichlet Unit Theorem}, 
the group of units of $T$ defined over any number field $K$ has a decomposition $T_\BZ \cong W \times E$  
on which the group $E$ is isomorphic to $\BZ^{r_\iy-r_K}$ and $|W| = w_T$. 
Hence Ono's result can be generalized to an algebraic torus defined over any number field $K$, 
by including the discriminant which is different from 1 in the general case. 
Recall that 
$$ \om_T = |\Delta_k|^{-d/2} \om_\iy \prod\limits_\fp L_\fp(1,\chi_T)\om_\fp. $$
Hence
\begin{align} \label{generalized Ono invariant}
c_T^\Ono &= \fc{\om_T}{\Om_T/t_K} (\Phi^{-1}(P)) = \fc{R_T \om_T(\Phi^{-1}(I))} {\Om_T/t_K(\Phi^{-1}(P))} 
          = \fc{R_T |\Delta_k|^{-d/2} \int_{M_\iy} \om_\iy \prod _\fp L_\fp(1,\chi_T)\om_\fp }{w_T/h_T} \\
\nonumber         &= \fc{R_T h_T}{w_T} |\Delta_k|^{-d/2} \int_{M_\iy} \om_\iy \prod _\fp \int_{T_\fp(\CO_\fp)} L_\fp(1,\chi_T)\om_\fp. 
\end{align}

\begin{remark}
This generalization is not required in case (F) on which the group of units of $T$ has no free part, 
and thus has no regulator. 
Indeed, Ono's result in that case is not restricted to $\BF_q(x)$ being the analogue of $\BQ$, 
and Ono's formula does include the discriminant there. 
We will write this formula explicitly in Section (\ref{Quasi-Discriminant of Tori defined over Algebraic Function Fields}). 
\end{remark}

%----------------------------------
\subsubsection{Shyr's invariant} \label{Shyr approach}
%----------------------------------

J. M. Shyr in \cite{Shyr1} gave a similar definition to the one of Ono in the case of algebraic $\BQ$-tori. 
Consider the decomposition of a local measure $d$ into $d T_v(K_v) = d(T_v(K_v)/T_v(\CO_v)) \cdot d T_v(\CO_v)$. 
The measure $\nu_v$ is defined by the pullback of the 
measure $\fc{dx_1 \cdot \cdot \cdot dx_{r_v}}{x_1 \cdot \cdot \cdot x_{r_v}}$ on $\BR_+^{r_v}$ 
(and the canonical discrete measure on $\BZ^{r_\fp}$, respectively) for $T_v(K_v)/T_v(\CO_v)$, 
matched together with the normalized Haar measure on $T_v(\CO_v)$. 
Then the measure $\nu_T$ is defined by the infinite product $\prod_v \nu_v$ and $\nu_T(T_A/T_K) = \nu_T (T_A/T_A^c)$. 
In this new construction, the other arithmetic invariants are taken from Ono's definition. 
Explicitly global Shyr invariant is computed by:
\begin{equation*} \label{Shyr invariant computation in number fields}
c_T^\Shyr = \int_{M_\iy} \om_\iy \prod_\fp \int_{T_\fp(\CO_\fp)} L_\fp(1,\chi_T) \om_\fp. 
\end{equation*} 
Shyr obtained a relation between the Haar measures: $\tau_T = \rho_T^{-1} c_T^{\Shyr} \nu_T$. 
This led him to a formula reflecting the relation between the other arithmetic invariants of $T$ 
which can be viewed as a torus analogue of the class number formula, namely:
\begin{equation} \label{Tori class number formula}
c_T^\Shyr = \fc{ \rho_T \tau_T w_T}{h_T R_T}. 
\end{equation}

As we have done in (\ref{generalized Ono invariant}), due to the generalization of the unit Theorem, 
this result can be generalized to an algebraic torus $T$ defined over any number field $K$ 
on which the field discriminant may not be $1$. 
In that general case we would get
\begin{equation} \label{global Shyr invariant}
c_T^\Shyr = \om_T(\Phi^{-1}(I))
            = |\Delta_K|^{-d/2} \int_{M_\iy} \om_\iy \prod_\fp \int_{T_\fp(\CO_\fp)} L_\fp(1,\chi_T) \om_\fp. 
\end{equation} 

The number $D_T = 1/ (c_T^\Shyr)^2$ is called the \emph{quasi-discriminant} of $T$ over $K$.

\subsection{Main Theorem}
%------------------------

With the above notations, using our previous local results, 
namely, Corollaries \ref{local Shyr invariant computation} 
and \ref{local Shyr invariant using cocharacter group}, 
we now get to the following computation of the Shyr invariant 
as appears in (\ref{global Shyr invariant}) with the relation in (\ref{Tori class number formula}).   

\begin{theorem} \label{Theorem1}
For any prime $\fp$ of $K$, let $F_\fp$ be the local Frobenius automorphism, 
let $X_\bullet(T_\fp)$ be the cocharacter group of $T_\fp = T \otimes K_\fp$ 
and let $X_\bullet(T_\fp)_{I_\fp}$ be its coinvariants factor group. 
%and let $\G_\fp$ be the decomposition group with inertia subgroup $I_\fp$. 
Then:
$$ c_T^\Shyr = |\Delta_K|^{-d/2} C_\iy \prod\limits_\fp \left| \ker(1-F_\fp|X_\bullet(T_\fp)_{I_\fp})_\tor \right| 
             = \fc{\rho_T \tau_T w_T }{ h_T R_T }. $$
where
\begin{equation*}
C_\iy = \left \{ \begin{array}{l l}
\int_{M_\iy} \om_\iy & \text{case(N)} \\ \\
(\ln q)^{-r_K}       & \text{case(F)} \\
\end{array}\right.
\end{equation*}
and $\Delta_K = q^{2g-2}, \ R_T=1$ in case(F).
\end{theorem}

\bk

%-----------------------------
\section{Appendix: Tori defined over algebraic function fields} 
\label{Quasi-Discriminant of Tori defined over Algebraic Function Fields}
%-----------------------------

Let $X$ be a smooth, projective and irreducible algebraic curve of genus $g$ defined over $\BF_q$  
and let $Y$ be a finite Galois cover of $X$. 
Let $K = \BF_q(X)$ and $L = \BF_q(Y)$ be the corresponding fields of rational functions. 
Then $L/K$ is a finite Galois extension with $\G = Gal(L/K)$. 
Let $T \in \C(L/K)$ be an algebraic torus of dimension $d$. 

\bk

Just as for number fields, let $\om_\fp$ be a normalized invariant form on $K_\fp$, 
i.e., such that $\om_p(\CO_\fp) = 1$. 
Since $\deg \fp = (\CO_\fp / \fp : \BF_q)$, we have $|\CO_\fp / \fp| = q^{\deg \fp}$. 
Hence the normalization condition implies $\om_\fp(\fp) = q^{-\deg \fp}$. 
The infinite product $\prod_\fp \om_\fp$ 
which is multiplied by the set $L_\fp(1,\chi_T,L/K)$ as a system of correcting factors, 
induces the Tamagawa measure on $T$ (see \cite[\S~2.2]{Weil}):
$$ \om_T = q^{-(g-1)d} \prod_\fp L_\fp (1,\chi_T) \om_ \fp, \ \ d = \dim T. $$

Define the normalized measure $\Om_T$ on $T_A$ by the condition $ \int_{T_A^1/T_K} d\Om_T = 1. $ \\
Then $\om_T = (\ln q)^{r_K} c_T^\Ono \Om_T$,  reflecting the decomposition
$$ d(T_A/T_K) = d(T_A/T_A^1) \cdot d(T_A^1/T_K) $$
on which $T_A/T_A^1$ is given the measure $(\ln q)^{r_K}$ (see \cite[\S~3.2]{Ono}). 
Hence in this case we get the relation:
\begin{equation} \label{Ono relation}
 \tau_T = \fc{c_T^\Ono }{ \rho_T } 
        = \fc{h_T \prod_\fp \int_{T_\fp(\CO_\fp)} L_\fp(1,\chi_T) \om_\fp }{ w_T \rho_T (\ln q)^{r_K} q^{d(g-1)} }. 
\end{equation}

As in Shyr's approach, consider the decomposition
$$ T_A/T_K \cong \left( T_A/T_A^1 \right) \times \left( T_A^1 / T_A^{S_\iy}T_K \right) \times \left( T_A^{S_\iy}T_K / T_K \right). $$
By the same construction as in (\ref{Shyr approach}), $\mu_T|_{T_A/T_A^1} = t_K$ and therefore:
$$ \rho_T^{-1} \om_T (T_A/T_K) = t_K(T_A/T_A^1) \cdot \tau_T = \fc{\mu_T(T_A/T_K)}{\mu_T(T_A/T_A^1)} \cdot \tau_T. $$
Now, as we gave the measure $(\ln q)^{r_K}$ to each point in $T_A/T_A^1$, we get
$$ \rho_T^{-1} \om_T (T_A/T_K) = \tau_T \cdot \fc{\mu_T(T_A/T_K)}{(\ln q)^{r_K}}. $$ 
Thus over $T_A/T_K$ we have $\rho_T^{-1} \om_T \cdot (\ln q)^{r_K} = \tau_T \mu_T $. \\
Since here there are no archimedean places, we have $T_A^{S_\iy} = \prod_\fp T_\fp(\CO_\fp) = T_A^{1,S_\iy}$. \\
Now $h_T = (T_A^1:T_A^{S_\iy}T_K)$ and $T_A^{S_\iy}T_K / T_K \cong T_A^{S_\iy} / W$ where $W$ is finite. \\
Thus $c_T^\Shyr = \fc{\om_T}{\nu_T}(T_A^{S_\iy})$. 
But: $T_A^{S_\iy} \subset T_A^1$ and so: $\nu_T(T_A^{S_\iy}) = 1 $. 
Hence
\begin{equation} \label{Shyr global invariant definition}
 c_T^\Shyr = \fc{\om_T}{\nu_T} (T_A^{S_\iy}) 
           = \fc{ (\ln q)^{-r_K} }{ q^{d(g-1)} } \prod\limits_\fp L_\fp(1, \chi_T) \om_\fp(T_\fp(\CO_\fp)). 
\end{equation}
From the equation above $\tau_T = \rho_T^{-1} c_T^\Shyr \nu_T$, we get the following relation, 
which can be viewed as a \emph{class number formula analogue} 
for algebraic tori defined over function fields:
\begin{equation} \label{class number formula analogue}
c_T^\Shyr = \fc{\om_T}{\nu_T} (T_A^{S_\iy}) = \rho_T \tau_T \mu_T (T_A^{S_\iy}) = \fc{\rho_T \tau_T w_T }{ h_T }. 
\end{equation}
Note that this result is none other than formula (\ref{Ono relation}) obtained by T. Ono (see \cite[3.8.10']{Ono}). 

%\begin{cor} \label{tori product shyr invariant}
%Let $T_1,T_2$ be algebraic tori defined over an algebraic function field and let $T = T_1 \times T_2$.
%Then $c_T^{\Shyr} = c_{T_1}^{\Shyr} \cdot c_{T_2}^{\Shyr}$.  
%\end{cor}

%\begin{proof} 
%Comes from (\ref{class number formula analogue}) and the direct decomposition of $\rho_T, h_T, w_T, \tau_T$.
%\end{proof}

\bk

\end{document}